\documentclass[12pt]{article}
\usepackage{mathrsfs}
\usepackage{amssymb}
\usepackage[hang,flushmargin]{footmisc}
\usepackage{multirow}
\usepackage{amsfonts}
\usepackage{graphicx}
\usepackage{subfigure}
\usepackage{amsmath,amsthm,amssymb,latexsym}
\usepackage[colorlinks, linkcolor=blue, anchorcolor=gree, citecolor=red]{hyperref}
\usepackage[numbers,sort&compress]{natbib}
\usepackage{xcolor}
\usepackage{enumerate}
\usepackage{bm}
 \makeatletter
    
    \newcommand{\Rmnum}[1]
    {\expandafter\@slowromancap\romannumeral #1@}
    \makeatother
\def\wz{\tilde}

\newtheorem{thm}{Theorem}[section]

\newtheorem{lemma}[thm]{Lemma}
\newcounter{foo}[subsection]

\newtheorem{step}[foo]{Step}

\newtheorem{example}[thm]{Example}
\newtheorem{defin}[thm]{Definition}

\newtheorem{fact}[thm]{Fact}

\theoremstyle{definition}
\newtheorem{remark}[thm]{Remark}

\usepackage{CJK}
\begin{document}
\begin{CJK*}{GBK}{song}

\renewcommand{\baselinestretch}{1.3}
\title{$P$-polynomial weakly distance-regular digraphs}

\author{Qing Zeng\textsuperscript{a}\quad Yuefeng Yang\textsuperscript{b}\footnote{\scriptsize Corresponding author.}
\quad Kaishun Wang\textsuperscript{a}
\\
\\{\footnotesize  \textsuperscript{a} \em  Laboratory of Mathematics and Complex Systems (MOE),}  \\
{\footnotesize  \em School of Mathematical Sciences, Beijing Normal University, Beijing, 100875, China }\\
{\footnotesize  \textsuperscript{b} \em School of Science, China University of Geosciences, Beijing, 100083, China}}
\date{}
\maketitle
\footnote{\scriptsize\qquad{\em E-mail address:} ~~qingz@mail.bnu.edu.cn ~~(Qing Zeng), ~~yangyf@cugb.edu.cn ~~(Yuefeng Yang), wangks@bnu.edu.cn ~~(Kaishun Wang).}

\begin{abstract}

A weakly distance-regular digraph is $P$-polynomial if its attached scheme is $P$-polynomial. In this paper, we characterize all $P$-polynomial weakly distance-regular digraphs.
\medskip

\noindent {\em AMS classification:} 05E30

\noindent {\em Key words:} $P$-polynomial association scheme; distance-regular digraph; weakly distance-regular digraph
\end{abstract}

\section{Introduction}

All the digraphs considered in this paper are finite, simple, strongly connected and not undirected. For a digraph $\Gamma$, we write $V\Gamma$ and $A\Gamma$ for the vertex set and the arc set of $\Gamma$, respectively. In $\Gamma$, a \emph{path} of length $r$ from $x$ to $y$ is a finite sequence of vertices $(x=w_{0},w_{1},\ldots,w_{r}=y)$ such that $(w_{t-1}, w_{t})\in A(\Gamma)$ for $1\leq t\leq r$. The length of a shortest path from $x$ to $y$ is called
the \emph{distance} from $x$ to $y$ in $\Gamma$, denoted by $\partial_\Gamma(x,y)$. The maximum value of distance function in $\Gamma$ is called the \emph{diameter} of $\Gamma$. We define $\Gamma_{i}$ ($0\leq i\leq d$) to be the set of ordered pairs $(x,y)$ with $\partial_{\Gamma}(x,y)=i$, where $d$ is the diameter of $\Gamma$. A path $(w_{0},w_{1},\ldots,w_{r-1})$ is called a \emph{circuit} of length $r$ when $(w_{r-1},w_0)\in A\Gamma$. The \emph{girth} of $\Gamma$ is the length of a shortest circuit in $\Gamma$. Let $\wz{\partial}_{\Gamma}(x,y):=(\partial_{\Gamma}(x,y),\partial_{\Gamma}(y,x))$ be the \emph{two-way distance} from $x$ to $y$ in $\Gamma$, and $\wz{\partial}(\Gamma)$ the set of all pairs $\wz{\partial}_{\Gamma}(x,y)$. For any $\wz{i}\in\wz{\partial}(\Gamma)$, we define $\Gamma_{\wz{i}}$ to be the set of ordered pairs $(x,y)$ with $\wz{\partial}_{\Gamma}(x,y)=\wz{i}$.

As a natural directed version of distance-regular graphs, Wang and Suzuki \cite{KSW03} introduced the concept of weakly distance-regular digraphs. A digraph $\Gamma$ is said to be \emph{weakly distance-regular} if, for any $\wz{h}$, $\wz{i}$, $\wz{j}\in\wz{\partial}(\Gamma)$ and $\wz{\partial}_{\Gamma}(x,y)=\wz{h}$, the number of $z\in V\Gamma$ such that $\wz{\partial}_{\Gamma}(x,z)=\wz{i}$ and $\wz{\partial}_{\Gamma}(z,y)=\wz{j}$ depends only on $\wz{h}$, $\wz{i}$, $\wz{j}$. In other words, $\Gamma$ is weakly distance-regular if $\mathfrak{X}(\Gamma)=(V\Gamma,\{\Gamma_{\wz{i}}\}_{\wz{i}\in\wz{\partial}(\Gamma)})$ is an association scheme. We call $\mathfrak{X}(\Gamma)$ the \emph{attached scheme} of $\Gamma$. A weakly distance-regular digraph is \emph{$P$-polynomial} if its attached scheme is $P$-polynomial.

Since 2003, some special families of weakly distance-regular digraphs have been classified. See \cite{KSW03,HS04,KSW04,LGG11,YYF16,YYF18} for small valency, \cite{HS04} for thin case, \cite{YYF20} for quasi-thin case, and \cite{YYF} for thick case.


In this paper, we study $P$-polynomial weakly distance-regular digraphs, and obtain the following result.

\begin{thm}\label{main}
Let $\Gamma$ be a weakly distance-regular digraph whose attached scheme $\mathfrak{X}=(X,\{R_i\}_{i=0}^d)$ is $P$-polynomial with respect to the ordering $R_0,R_1,\ldots,R_d$. Then $\Gamma$ is isomorphic to one of the following digraphs:
\begin{itemize}
\item [{\rm(i)}] $(X,R_1)$ or $(X,R_{g-1})$;

\item[{\rm(ii)}] $(X,R_2)$ or $(X,R_{g-2})$, where $k_1>k_g+1$ and $g\in\{6,8\}$;

\item[{\rm(iii)}] $(X,R_1\cup R_2)$ or $(X,R_{g-2}\cup R_{g-1})$, where $2\mid g$;

\item [{\rm (iv)}] $(X,R_1\cup R_g)$ or $(X,R_{g-1}\cup R_g)$, where $d=g$;

\item[{\rm(v)}] $(X,R_2\cup R_g)$ or $(X,R_{g-2}\cup R_g)$, where $k_1>k_g+1$, $d=g$ and $g\in\{6,8\}$;

\item [{\rm(vi)}] $(X,R_1\cup R_2\cup R_g)$ or $(X,R_{g-2}\cup R_{g-1}\cup R_g)$, where $d=g$, $2\mid g$ and $g>4$.
\end{itemize}
Here, $k_i$ is the valency of the relation $R_i$ for $i\in\{1,g\}$ and $g$ is the girth of $(X, R_1)$.
\end{thm}

The examples of non-symmetric $P$-polynomial association schemes are very few, we do not know whether all the cases mentioned above are actually realized. One can verify the cases (i) and (iv) are realized. The case (iii) is realized when $(X, R_1)$ is one of  the examples in  \cite{LM}. The case (vi) is realized when $(X, R_1)$ is a lexicographic product of a directed cycle by a complete graph. We do not know at present whether the case (ii) or the case (v) is actually realized.


The remaining of this paper is organized as follows. In Section 2, we provide the required concepts and notations about association schemes. In Section 3, we recall the definition of distance-regular digraphs, and prove some results which are used frequently in this paper. In Section 4, we give a proof of Theorem \ref{main}.

\section{Association schemes}

In this section, we present some concepts and notations of association schemes.

A \emph{$d$-class association scheme} $\mathfrak{X}$ is a pair $(X,\{R_{i}\}_{i=0}^{d})$, where $X$ is a finite set, and each $R_{i}$ is a
nonempty subset of $X\times X$ satisfying the following axioms (see \cite{EB84,PHZ96,PHZ05} for a background of the theory of association schemes):
\begin{itemize}
\item [{\rm(i)}] $R_{0}=\{(x,x)\mid x\in X\}$ is the diagonal relation;

\item [{\rm(ii)}] $X\times X=R_{0}\cup R_{1}\cup\cdots\cup R_{d}$, $R_{i}\cap R_{j}=\emptyset~(i\neq j)$;

\item [{\rm(iii)}] for each $i$, $R_{i}^{\rm T}=R_{i^{*}}$ for some $0\leq i^{*}\leq d$, where $R_{i}^{\rm T}=\{(y,x)\mid(x,y)\in R_{i}\}$;

\item [{\rm(iv)}] for all $i,j,h$, the cardinality of the set $$P_{i,j}(x,y):=R_{i}(x)\cap R_{j^{*}}(y)$$ is constant whenever $(x,y)\in R_{h}$, where $R(x)=\{y\mid (x,y)\in R\}$ for $R\subseteq X\times X$ and $x\in X$. This constant is denoted by $p_{i,j}^{h}$.
\end{itemize}
The integers $p_{i,j}^{h}$ are called the \emph{intersection numbers} of $\mathfrak{X}$. We say that $\mathfrak{X}$ is \emph{commutative} if $p_{i,j}^{h}=p_{j,i}^{h}$ for all $0\leq i,j,h\leq d$. The subsets $R_{i}$ are called the \emph{relations} of $\mathfrak{X}$. For each $i$, the integer $k_{i}$ $(=p_{i,i^{*}}^{0})$ is called the \emph{valency} of $R_{i}$. A relation $R_{i}$ is called \emph{symmetric} if $i=i^{*}$, and \emph{non-symmetric} otherwise. An association scheme is called \emph{symmetric} if all relations are symmetric, and \emph{non-symmetric} otherwise.

Let $\mathfrak{X}=(X,\{R_{i}\}_{i=0}^{d})$ be an association scheme with $|X|=n$. The \emph{adjacency matrix} $A_{i}$ of $R_{i}$ is the $n\times n$ matrix whose $(x,y)$-entry is $1$ if $(x,y)\in R_{i}$, and $0$ otherwise. 
We say that $\mathfrak{X}$ is a \emph{$P$-polynomial} association scheme 
with respect to the ordering $R_0,R_1,\ldots,R_d$, if there exist some complex coefficient polynomials $v_{i}(x)$ of degree $i$ ($0\leq i\leq d$) such that $A_i=v_{i}(A_1).$

We close this section with a property of intersection numbers which will be used frequently in the remainder of this paper.

\begin{lemma}\label{jcs}
{\rm (\cite[Section 2, Proposition 2.2]{EB84})} Let $\mathfrak{X}=(X,\{R_{i}\}_{i=0}^{d})$ be an association scheme. The following hold:
\begin{itemize}
\item [{\rm(i)}] $p^{h}_{i,j}k_{h}=p^{i}_{h,j^{*}}k_{i}=p^{j}_{i^{*},h}k_{j}$;

\item [{\rm(ii)}] $\sum^{d}_{r=0}p^{r}_{e,l}p^{h}_{m,r}=\sum^{d}_{t=0}p^{t}_{m,e}p^{h}_{t,l}$.
\end{itemize}
\end{lemma}

\section{Distance-regular digraphs}

As a directed version of distance-regular graph, Damerell \cite{RMD81} introduced the concept of distance-regular digraphs. A digraph $\Gamma$ of diameter $d$ is said to be \emph{distance-regular} if $\mathfrak{X}(\Gamma)=(V\Gamma,\{\Gamma_{i}\}_{0\leq i\leq d})$ is a $P$-polynomial non-symmetric association scheme with respect to the ordering $\Gamma_0,\Gamma_1,\ldots,\Gamma_d$.
Moreover, every $P$-polynomial non-symmetric association scheme arises from a distance-regular digraph in this way. 

A digraph $\Gamma$ of girth $g$ is \emph{stable} if $0<\partial_{\Gamma}(x,y)<g$ implies that $\partial_{\Gamma}(x,y)+\partial_{\Gamma}(y,x)=g.$
In \cite{RMD81}, Damerell proved that a distance-regular digraph is stable, and presented  the following basic results.

\begin{thm}\label{drdg}
{\rm (\cite[Theorems 2 and 4]{RMD81})} Let $\Gamma$ be a distance-regular digraph of diameter $d$ and girth $g$. Then $d=g-1$ (short type) or $d=g$ (long type). Moreover, if $d=g$, then $\Gamma$ is a lexicographic product of a distance-regular digraph of diameter $g-1$ by an empty graph.
\end{thm}

In view of \cite{DL89} and the stability of distance-regular digraphs, we say that the \emph{girth} of a $P$-polynomial non-symmetric association scheme $(V\Gamma,\{\Gamma_{i}\}_{0\leq i\leq d})$ with respect to the ordering $\Gamma_0,\Gamma_1,\ldots,\Gamma_d$ is the girth of $(V\Gamma,\Gamma_1)$.

Enomoto and Mena \cite{EM} showed that two one-parameter families of distance-regular digraphs of girth $4$ could possibly exist. Subsequently, Liebler and Mena \cite{LM} gave an infinite one-parameter family of distance-regular digraphs over an extension ring of $\mathbb{Z}/4\mathbb{Z}$.
In 1993, Leonard and Nomura \cite{ln} proved that except directed cycles all short distance-regular digraphs  have girth not more than $8$.


In the remaining of this section, we always assume that $\Gamma$ is a distance-regular digraph of valency $k_1>k_g+1$ with diameter $d$ and  girth $g$. Next we give some results concerning such digraphs which will be used frequently in this paper.



%
%
%
%

\begin{lemma}\label{p122}
If $1\leq i\leq g-1,$ then $p^{1}_{2,i}\neq0$.
\end{lemma}
\begin{proof}
In view of Theorem \ref{drdg},
we only need to consider the case $d=g-1$. Since $k_g=0$, one gets $k_1>1$. By setting $(h,m,e,l)=(g-i,1,1,g-1)$ in Lemma \ref{jcs} (ii), we have
\begin{align}
&p^{g-i-1}_{1,g-1}p^{g-i}_{1,g-i-1}+p^{g-1}_{1,g-1}p^{g-i}_{1,g-1}\leq p^{1}_{1,1}p^{g-i}_{1,g-1}+p^{2}_{1,1}p^{g-i}_{2,g-1}.\label{eq1}
\end{align}
Since $\Gamma$ is stable, we obtain $p^{g-1}_{1,g-1}=p^{1}_{1,1}$ from Lemma \ref{jcs} (i), and so $$p^{g-i-1}_{1,g-1}p^{g-i}_{1,g-i-1}\leq p^{2}_{1,1}p^{g-i}_{2,g-1}$$ from \eqref{eq1}. Note that $k_1>1$. In view of \cite[Lemma 2.1]{ln}, one has $p^{g-i-1}_{1,g-1}\neq0$, and so $p^{g-i}_{2,g-1}\neq0$. By Lemma \ref{jcs} (i), and the stability and commutativity of $\Gamma$, we have $$k_1p^{1}_{2,i}=k_{g-i}p^{g-i}_{2,g-1}\neq0.$$ This completes the proof of this lemma.
\end{proof}

The stability and commutativity of $\Gamma$ will be used frequently in the sequel, so we no longer refer to it for the sake of simplicity.

\begin{lemma}\label{pi}
If $0\leq i\leq\min\{3,g-3\}$, then $p^{i}_{1,i+1}<p^{i}_{g-1,i+1}$.
\end{lemma}
\begin{proof}
In view of Theorem \ref{drdg},
we only need to consider the case $d=g-1$.  It is valid when $i=0$. By setting $(h,m,e,l)=(i,1,i,g-1)$ in Lemma \ref{jcs} (ii), we obtain $$\sum_{r=i-1}^{g-1}p^{r}_{i,g-1}p^{i}_{1,r}=\sum_{t=1}^{i+1}p^{t}_{1,i}p^{i}_{t,g-1}.$$
According to Lemma \ref{jcs} (i), one gets, for $i-1\leq r<g$ and $1\leq t\leq i+1$, $$k_rp^{r}_{i,g-1}=k_{i}p^{i}_{1,r}~\mbox{and}~ k_tp^{t}_{1,i}=k_ip^{i}_{g-1,t},$$ which imply $$\sum_{r=i-1}^{g-1}(p^{i}_{1,r})^{2}/k_r=\sum_{t=1}^{i+1}(p^{i}_{t,g-1})^{2}/k_t.$$ Since $p^{i}_{1,g-1}=p^{i}_{g-1,1}$ and $i+1<g-1$, we get
\begin{align}
&(p^{i}_{1,i-1})^{2}/k_{i-1}+(p^{i}_{1,i})^{2}/k_{i}+(p^{i}_{1,i+1})^{2}/k_{i+1}\leq\sum_{t=2}^{i+1}(p^{i}_{g-1,t})^{2}/k_{t}.\label{eq2}
\end{align}
If $i=1$, from \eqref{eq2}, then $$(p^{1}_{1,0})^{2}+(p^{1}_{1,1})^{2}/k_{1}+(p^{1}_{1,2})^{2}/k_{2}\leq (p^{1}_{g-1,2})^{2}/k_{2},$$ and so $p^{1}_{1,2}<p^{1}_{g-1,2}$ since $p^{1}_{1,0}\neq0$. In view of Lemma \ref{jcs} (i), we have $p^{i}_{1,i}=p^{i}_{g-1,i}$. If $i=2$, from \eqref{eq2}, then $$(p^{2}_{1,1})^{2}/k_{1}+(p^{2}_{1,3})^{2}/k_{3}\leq(p^{2}_{g-1,3})^{2}/k_{3},$$ and so $p^{2}_{1,3}<p^{2}_{g-1,3}$ since $p^{2}_{1,1}\neq0$.

We only need to consider the case $i=3$. By Lemma \ref{jcs} (i), we get $$p_{1,3}^2k_2=p^{3}_{g-1,2}k_3~\mbox{and}~p^{2}_{g-1,3}k_2=p^{3}_{1,2}k_3.$$
The fact  $p^{2}_{1,3}<p^{2}_{g-1,3}$ implies $p^{3}_{g-1,2}<p^{3}_{1,2}$. Since $p^{3}_{1,3}=p^{3}_{g-1,3}$, from \eqref{eq2}, one has $$(p^{3}_{1,2})^{2}/k_2+(p^{3}_{1,4})^{2}/k_4\leq(p^{3}_{g-1,2})^{2}/k_2+(p^{3}_{g-1,4})^{2}/k_4.$$ It follows that $p^{3}_{1,4}<p^{3}_{g-1,4}$.
\end{proof}

\begin{lemma}\label{ai}
If $0\leq i\leq\min\{2,g/2-1\}$, then $p^{i}_{1,i}<p^{i+1}_{1,i+1}$.
\end{lemma}
\begin{proof}
According to  \cite[Lemma 2.1]{ln} and Lemma \ref{drdg}, we have $p^{1}_{1,1}\neq0$, and so $0=p^{0}_{1,0}<p^{1}_{1,1}$. Thus, the case $i=0$ is valid. Now suppose $i>0$. It follows that $g\geq4$. By setting $(h,m,e,l)=(i+1,1,i,g-1)$ in Lemma \ref{jcs} (ii), we obtain $$\sum^{g-1}_{r=i}p^{r}_{i,g-1}p^{i+1}_{1,r}=\sum_{t=1}^{i+1}p^{t}_{1,i}p^{i+1}_{t,g-1}.$$
In view of Lemma \ref{jcs} (i), one has $p^{g-1}_{i,g-1}=p^{1}_{1,i}$ and $$k_2k_{i+1}p^{g-2}_{i,g-1}p^{i+1}_{1,g-2}=k_{1}^{2}p^{1}_{2,i}p^{1}_{2,i+1}\neq0$$ from Lemma \ref{p122}, which imply
\begin{align}
&\sum^{g-3}_{r=i}p^{r}_{i,g-1}p^{i+1}_{1,r}<\sum_{t=2}^{i+1}p^{t}_{1,i}p^{i+1}_{t,g-1}.\label{eq3}
\end{align}
According to Lemma \ref{jcs} (i), we have $p^{i}_{i,g-1}=p^{i}_{1,i}$. If $i=1$, from \eqref{eq3}, then $p^{1}_{1,1}p^{2}_{1,1}<p^{2}_{1,1}p^{2}_{1,2}$, which implies $p^{1}_{1,1}<p^{2}_{1,2}$. By Lemma \ref{pi}, one gets $p^{2}_{1,3}<p^{2}_{3,g-1}$, and so $p^{3}_{2,g-1}<p^{3}_{1,2}$ from Lemma \ref{jcs} (i). If $i=2$, then $$p^{2}_{1,2}(p^{3}_{1,2}-p^{3}_{2,g-1})<p^{3}_{1,3}(p^{3}_{1,2}-p^{3}_{2,g-1})$$  from \eqref{eq3}, which implies $p^{2}_{1,2}<p^{3}_{1,3}$. This completes the proof of this lemma.
\end{proof}

\begin{lemma}\label{p1ij}
If $0<i\leq(g+1)/2$, then $p^{1}_{i,i}\neq0$.
\end{lemma}
\begin{proof}
In view of Theorem \ref{drdg}, we only need to consider the case $d=g-1$. It follows that $k_1>1$.
According to \cite[Theorem 3.3]{ln}, we have $g\leq8$, which implies $i\leq4$ since $i\leq(g+1)/2$. The case of $i=1$ is valid from \cite[Lemma 2.1]{ln}, and the case of $i=2$ is also valid from Lemma \ref{p122}.

Suppose $i=3$. By setting $(h,m,e,l)=(g-3,1,2,g-1)$ in Lemma \ref{jcs} (ii), we have $$p^{g-3}_{2,g-1}p^{g-3}_{1,g-3}+p^{g-1}_{2,g-1}p^{g-3}_{1,g-1}\leq p^{1}_{1,2}p^{g-3}_{1,g-1}+p^{2}_{1,2}p^{g-3}_{2,g-1}+p^{3}_{1,2}p^{g-3}_{3,g-1}.$$ In view of Lemma \ref{jcs} (i), we get $$p^{g-1}_{2,g-1}=p^{1}_{1,2}~\mbox{and}~p^{g-3}_{1,g-3}=p^{3}_{1,3},$$ which imply
\begin{align}
&p^{g-3}_{2,g-1}p^{3}_{1,3}\leq p^{2}_{1,2}p^{g-3}_{2,g-1}+p^{3}_{1,2}p^{g-3}_{3,g-1}.\label{eq4}
\end{align}
According to Lemma \ref{jcs} (i), we obtain $$k_3p^{g-3}_{2,g-1}=k_{1}p^{1}_{2,3}\neq0$$ from Lemma \ref{p122}. By Lemma \ref{ai}, one has $p^{2}_{1,2}<p^{3}_{1,3}$. It follows from \eqref{eq4} that $p^{g-3}_{3,g-1}\neq0$. Lemma \ref{jcs} (i) implies $p^{1}_{3,3}\neq0$.

Suppose $i=4$. Since $i\leq(g+1)/2$, we have $g=7$ or $8$. If $g=7$, then $p^{1}_{4,4}=p^{6}_{3,3}\neq0$ from Lemma \ref{jcs} (i). Now we consider the case $g=8$. Assume the contrary, namely, $p^{1}_{4,4}=0$. Pick $(x,y)\in\Gamma_3$. According to Lemma \ref{ai}, we have $p^{3}_{1,3}\neq0$. By Lemma \ref{jcs} (i), one gets $$k_3p^{3}_{4,g-1}=k_4p^{4}_{1,3}\neq0.$$ Let $z\in P_{1,3}(x,y)$ and $w\in P_{4,g-1}(x,y)$. Since $p^{3}_{1,2}\neq0$, there exists $w'\in P_{1,2}(x,y)$. The fact $y\in P_{2,1}(w',w)$ implies $w'\in P_{1,3}(x,w)$. By $p^{1}_{4,4}=0$, we have $z\in P_{1,3}(x,w)$, and so $w\in P_{3,g-1}(z,y)$. Since $z\in P_{1,3}(x,y)$ and $w\in P_{4,g-1}(x,y)$ were arbitrary, we have $$P_{1,3}(x,y)\cup\{w'\}\subseteq P_{1,3}(x,w)~ \mbox{and}~P_{4,g-1}(x,y)\subseteq P_{3,g-1}(z,y).$$ Then $p^{3}_{1,3}<p^{4}_{1,3}$ and $p^{3}_{4,g-1}\leq p^{3}_{3,g-1}$. By Lemma \ref{jcs} (i), we get $p^{3}_{3,g-1}=p^{3}_{1,3}$, and so $$k_4p^{3}_{1,3}<k_4p^{4}_{1,3}=k_3p^{3}_{4,g-1}\leq k_3p^{3}_{1,3}.$$ It follows that $k_4<k_3$, contrary to \cite[Lemma 1.1 (c)]{ln}.
\end{proof}

\begin{lemma}\label{p322}
If $1\leq i\leq\min\{4,g-1\}$, then $p^{i}_{2,2}\neq0$.
\end{lemma}
\begin{proof}
It is obvious that $p^4_{2,2}\neq0$. According to Lemma \ref{p122}, we have $p^{1}_{2,2}\neq0$. We only need to consider the case $i\in\{2,3\}$. By setting $(h,m,e,l)=(i,2,1,1)$ in Lemma \ref{jcs} (ii), we get $$p^{1}_{1,1}p^{i}_{2,1}+p^{2}_{1,1}p^{i}_{2,2}\geq p^{2}_{1,2}p^{i}_{2,1}.$$ In view of Lemma \ref{ai}, one gets $p^{1}_{1,1}<p^{2}_{1,2}$, and so $p^{i}_{2,1}\neq0$. Thus, $p^{i}_{2,2}\neq0$.
%
%
%
\end{proof}

\section{Proof of Theorem \ref{main}}

With notations in Theorem \ref{main}, we shall give a proof of this theorem in this section. Before that, we need some auxiliary facts and lemmas.



\begin{fact}\label{sub}
For all $1\leq i\leq d$, $R_{i}\subseteq A\Gamma$ or $R_{i}\cap A\Gamma=\emptyset$.
\end{fact}


\begin{fact}\label{zhuanzhi}
The digraph $(X, (A\Gamma)^{\rm T})$ is   weakly distance-regular, and has the same attached scheme with  $\Gamma$.
\end{fact}

\begin{lemma}\label{short and long}
Let $\Delta$ be a weakly distance-regular digraph of girth $g'$. Suppose that $\mathfrak{X}(\Delta)$ is a $P$-polynomial association scheme of girth $d+1$ with respect to the ordering $R_0,R_1,\ldots,R_d$. The following hold:
\begin{itemize}
\item [{\rm(i)}] If $2\mid d$ or $R_{(d+1)/2}\nsubseteq A\Delta,$ then a lexicographic product of $\Delta$ by a complete graph is a $P$-polynomial weakly distance-regular digraph;

\item [{\rm(ii)}] If $(g',g')\notin\wz{\partial}(\Delta)$, then a lexicographic product of $\Delta$ by an empty graph is a $P$-polynomial weakly distance-regular digraph.
\end{itemize}
\end{lemma}
\begin{proof}
(i)~Let $\Delta'$ be a lexicographic product of $\Delta$ by a complete graph. If  $2\mid d$, then $R_i$ is non-symmetric for  $1\leq i\leq d$; if $2\nmid d$, then   $R_{(d+1)/2}$ is the unique symmetric relation. Since $2\mid d$ or $R_{(d+1)/2}\nsubseteq A\Delta,$ by Fact \ref{sub}, we have $(1,1)\notin\wz{\partial}(\Delta)$. It follows from \cite[Proposition 2.6 (i)]{KSW03} that $\Delta'$ is weakly distance-regular.  By  \cite[Theorem 2.1]{AM91},  $\Delta'$ is $P$-polynomial. Thus, (i) holds.

(ii)~Let $\Delta''$ be a lexicographic product of $\Delta$ by an empty graph. Note that $(g',g')\notin\wz{\partial}(\Delta)$. In view of \cite[Proposition 2.4 (i)]{KSW03}, $\Delta''$ is weakly distance-regular.  By \cite[Theorem 2.1]{AM91}, $\Delta''$ is $P$-polynomial. Thus, (ii) holds.
\end{proof}

\begin{lemma}\label{sur}
Each digraph in Theorem \ref{main} (i)--(vi)   is a weakly distance-regular digraph with $\mathfrak{X}$ as its attached scheme.
\end{lemma}

\begin{proof}
According to Lemma \ref{short and long} (i), it suffices to show that each digraph in Theorem \ref{main} (i)--(iii) is a weakly distance-regular digraph with $\mathfrak{X}$ as its attached scheme.

Since the digraphs $(X,R_1)$ and $(X,R_{g-1})$ are both distance-regular, we only need to prove that the digraphs in Theorem \ref{main} (ii) and (iii) are weakly distance-regular digraphs with the attached scheme $\mathfrak{X}$. It follows that  $2\mid g$.  Let $(x_1,x_2,\ldots,x_{g/2}=x_0)$ be a sequence of elements such that $(x_i,x_{i+1})\in R_2$ with $0\leq i\leq g/2-1$.
According to Fact \ref{zhuanzhi}, we only need to prove that $(X,R_2)$ with $k_1>k_g+1$ and $g\in\{6,8\}$, and $(X,R_1\cup R_2)$ with $2\mid g$ are weakly distance-regular digraphs with the attached scheme $\mathfrak{X}$.

%

\begin{step}
Show that $(X,R_2)$ with $k_1>k_g+1$ and $g\in\{6,8\}$ is a weakly distance-regular digraph with the attached scheme $\mathfrak{X}$.
\end{step}

Let $\Delta=(X,R_2)$. Since $(X,R_1)$ is distance-regular, $(x_1,x_2,\ldots,x_{g/2})$ is a shortest circuit in $\Delta$, and so the girth of $\Delta$ is $g/2$.  By Lemma \ref{short and long} (ii), it suffices to show that $\Delta$ is a weakly distance-regular digraph with the attached scheme $\mathfrak{X}$ and $(g/2,g/2)\notin\wz{\partial}(\Delta)$ under the assumption $d+1=g$.

Since $A\Delta=R_2$, we have  $R_2=\Delta_{(1,g/2-1)}$. Pick $(x,y)\in R_1$. Since $p^{1}_{2,2}\neq0$ from Lemma \ref{p1ij}, we get $\partial_{\Delta}(x,y)=2$.

Suppose $g=6$. By Lemma \ref{jcs} (i) and Lemma \ref{p122}, one has $p^{5}_{2,4}=p^{1}_{2,4}\neq0$. Since $p^{5}_{2,2}=0$ and $p^{4}_{2,2}\neq0$, one gets $\partial_{\Delta}(y,x)=3$, and so $R_1\subseteq\Delta_{(2,3)}.$ In view of Lemma \ref{p322}, we obtain  $p_{2,2}^{g/2}\neq0$, which implies $R_{3}=\Delta_{(2,2)}$. It follows that $R_1=\Delta_{(2,3)}$. By the distance-regularity of $(X,R_1)$, $\Delta$ is a $P$-polynomial weakly distance-regular digraph with the attached scheme  $$(X,\{\Delta_{(0,0)},\Delta_{(1,2)},\Delta_{(2,1)},\Delta_{(2,2)},\Delta_{(2,3)},\Delta_{(3,2)}\}).$$

Suppose $g=8$. Choose $(x',y')\in R_3$. According to Lemma \ref{p322}, we have $\partial_{\Delta}(x',y')=2$. Since $p^{5}_{2,3}\neq0$ and $p^{5}_{2,2}=0$, one has $\partial_{\Delta}(y',x')=3$, and so $R_3\subseteq\Delta_{(2,3)}$. The fact $(y,x)\in R_7$ implies $\partial_{\Delta}(y,x)\geq4$. Since $p^{4}_{2,2}p^{7}_{3,4}\neq0$, we obtain $\partial_{\Delta}(y,x)=4$, and so $R_1\subseteq\Delta_{(2,4)}$. The fact that $p_{2,2}^{4}\neq0$ implies $R_{4}=\Delta_{(2,2)}$. It follows that $R_3=\Delta_{(2,3)}$ and $R_1=\Delta_{(2,4)}$. By the distance-regularity of $(X,R_1)$, $\Delta$ is a $P$-polynomial weakly distance-regular digraph with the attached scheme $$(X,\{\Delta_{(0,0)},\Delta_{(1,3)},\Delta_{(2,2)},\Delta_{(2,3)},\Delta_{(2,4)},\Delta_{(3,1)},\Delta_{(3,2)},\Delta_{(4,2)}\}).$$

\begin{step}
Show that $(X,R_1\cup R_2)$ with $2\mid g$ is a weakly distance-regular digraph with the attached scheme $\mathfrak{X}$.
\end{step}

Let $\Delta=(X,R_1\cup R_2)$. Since $(X,R_1)$ is  distance-regular, $(x_1,x_2,\ldots,x_{g/2})$ is a shortest circuit in $\Delta$, and so the girth of $\Delta$ is $g/2$. By Lemma \ref{short and long} (ii), it suffices to show that $\Delta$ is a weakly distance-regular digraph with the attached scheme $\mathfrak{X}$ and $(g/2,g/2)\notin\wz{\partial}(\Delta)$ under the assumption $d+1=g$.

Note that $k_g=0$. Suppose $k_1=1$. Since $(X,R_1)$ is  distance-regular,  $(X,R_1)$ is a directed cycle, which implies $k_i=1$ for $0\leq i\leq d$. It follows from \cite[Theorem 1.2 (ii)]{HS04} that $\Delta$ is a weakly distance-regular digraph  with the attached scheme $\mathfrak{X}$ and $(g/2,g/2)\notin\wz{\partial}(\Delta)$. Now we consider the case $k_1>1$.

Since $(x_1,x_2,\ldots,x_{g/2})$ is a shortest circuit in $\Delta$, we obtain $R_2\subseteq\Delta_{(1,g/2-1)}$. The fact that $p^{2}_{1,1}\neq0$ implies $R_1=\Delta_{(1,g/2)}$, and so $R_2=\Delta_{(1,g/2-1)}$.  If $g=4$, by the distance-regularity of $(X,R_1)$, then $\Delta$ is a $P$-polynomial weakly distance-regular digraph with the attached scheme $$(X,\{\Delta_{(0,0)},\Delta_{(1,1)},\Delta_{(1,2)},\Delta_{(2,1)}\}).$$ If $g=6$, then $R_{3}=\Delta_{(2,2)}$ since $p_{1,2}^{3}\neq0$, which implies that $\Delta$ is a $P$-polynomial weakly distance-regular digraph with the attached scheme $$(X,\{\Delta_{(0,0)},\Delta_{(1,2)},\Delta_{(1,3)},\Delta_{(2,1)},\Delta_{(2,2)},\Delta_{(3,1)}\}).$$



We only need to consider the case of $g=8$. Choose $(x,y)\in R_3$. Since $p^{3}_{1,2}\neq0$, we have $\partial_{\Delta}(x,y)=2$. The fact $(y,x)\in R_5$ implies $\partial_{\Delta}(y,x)\geq3$. Since $p^{5}_{2,3}\neq0$, one obtains $\partial_{\Delta}(y,x)=3$, which implies $R_3\subseteq\Delta_{(2,3)}$. By $p^{4}_{2,2}\neq0$, one has $R_4=\Delta_{(2,2)}$. It follows that $R_3=\Delta_{(2,3)}$. By the distance-regularity of $(X,R_1)$ again, $\Delta$ is a $P$-polynomial weakly distance-regular digraph with the attached scheme $$(X,\{\Delta_{(0,0)},\Delta_{(1,3)},\Delta_{(1,4)},\Delta_{(2,2)},\Delta_{(2,3)},\Delta_{(3,1)},\Delta_{(3,2)},\Delta_{(4,1)}\}).$$

This completes the proof of this lemma.
\end{proof}

\begin{lemma}\label{arc}
If $k_1>k_g+1$, then the number of $R_i$ satisfying $R_i\subseteq A\Gamma$ with $1\leq i\leq g-1$ is at most two.
\end{lemma}
\begin{proof}
Suppose for the contrary that the number of $R_i$ satisfying $R_i\subseteq A\Gamma$ with $1\leq i\leq g-1$ is at least three. According to  \cite[Theorem 3.3]{ln} and Theorem \ref{drdg}, we have $g\leq8$. Note that $\mathfrak{X}$ is the attached scheme of $\Gamma$. If $(R_i\cup R_{g-i})\cap A\Gamma\neq\emptyset$ for some $i\in\{1,2,\ldots,g-1\}$ with $i\neq g/2$, by Fact \ref{sub}, then $R_i\subseteq A\Gamma$ and $R_{g-i}\cap A\Gamma=\emptyset$, or $R_{g-i}\subseteq A\Gamma$ and $R_{i}\cap A\Gamma=\emptyset$ since $R_i$ is non-symmetric. It follows that $g\in\{6,7,8\}$. Pick $(x_i,y_i)\in R_i$ for $1\leq i\leq3$.

\textbf{Case 1.} $g=7$.

It is obviously that $(R_2\cup R_{5})\cap A\Gamma\neq\emptyset$. According to Facts \ref{sub} and \ref{zhuanzhi}, we may assume $R_2\subseteq A\Gamma$. Hence, $\partial_{\Gamma}(x_2,y_2)=1$. Note that $R_3\subseteq A\Gamma$ or $R_4\subseteq A\Gamma$. Suppose $R_3\subseteq A\Gamma$. Then $\partial_{\Gamma}(x_3,y_3)=1$. By $p^{4}_{2,2}\neq0$ and $p^{5}_{2,3}\neq0$, one has $\partial_{\Gamma}(y_3,x_3)\leq2$ and $\partial_{\Gamma}(y_2,x_2)\leq2$. Since $R_3$ and $R_2$ are non-symmetric, we have $R_3=\Gamma_{(1,2)}$ and $R_2=\Gamma_{(1,2)}$, a contradiction. Therefore, $R_{4}\subseteq A\Gamma$ and $\partial_{\Gamma}(y_3,x_3)=1$.

According to Lemma \ref{p322}, one gets $p^{3}_{2,2}\neq0$, which implies $\partial_{\Gamma}(x_3,y_3)\leq2$, and so $R_{4}=\Gamma_{(1,2)}$ since $R_{4}$ is non-symmetric. Note that $R_1\subseteq A\Gamma$ or $R_6\subseteq A\Gamma$. If $R_{1}\subseteq A\Gamma$, by $p^{5}_{1,4}\neq0$, then $R_{2}=\Gamma_{(1,2)}$ since $R_{2}$ is non-symmetric, which is impossible; if $R_{6}\subseteq A\Gamma$, by $p^{1}_{2,2}\neq0$ from Lemma \ref{p1ij}, then $R_{6}=\Gamma_{(1,2)}$ since $R_{6}$ is non-symmetric, a contradiction.

\textbf{Case 2.} $g=6$ or $8$.

\textbf{Case 2.1.} $(R_2\cup R_{g-2})\cap A\Gamma=\emptyset$.

Note that $g=8$ and $R_4\subseteq A\Gamma$. It is obvious that $(R_3\cup R_5)\cap A\Gamma\neq\emptyset$. By Facts \ref{sub} and \ref{zhuanzhi}, we may assume $R_3\subseteq A\Gamma.$ Note that $R_1\subseteq A\Gamma$ or $R_7\subseteq A\Gamma$. If $R_1\subseteq A\Gamma$, by $p^{5}_{1,4}\neq0$ and $p^{7}_{3,4}\neq0$, then $(x_3,y_3)$, $(x_1,y_1)\in\Gamma_{(1,2)}$, which implies $R_3=R_1=\Gamma_{(1,2)}$, a contradiction; if $R_7\subseteq A\Gamma$, by $p^{6}_{3,3}\neq0$ and $p^{2}_{3,7}=p^{6}_{1,5}\neq0$ from Lemma \ref{jcs} (i), then $R_2=\Gamma_{(2,2)}$, contrary to the fact that $R_2$ is non-symmetric.

\textbf{Case 2.2.} $(R_2\cup R_{g-2})\cap A\Gamma\neq\emptyset$.

By Facts \ref{sub} and \ref{zhuanzhi}, we may assume $R_2\subseteq A\Gamma$. Hence, $\partial_{\Gamma}(x_2,y_2)=1$.

Suppose $R_{g/2}\nsubseteq A\Gamma$. It is obvious that $g=8$ and $(R_3\cup R_{5})\cap A\Gamma\neq\emptyset$. If $R_3\subseteq A\Gamma$, by $p^{5}_{2,3}\neq0$ and $p^{6}_{3,3}\neq0$, then $(x_3,y_3)$, $(x_2,y_2)\in\Gamma_{(1,2)}$, which implies $R_3=R_2=\Gamma_{(1,2)}$, a contradiction. Thus, $R_{5}\subseteq A\Gamma$, and so $\partial_{\Gamma}(y_3,x_3)=1$.

According to Lemma \ref{p322}, one gets $p^{3}_{2,2}\neq0$, which implies $\partial_{\Gamma}(x_3,y_3)=2$. It follows that $R_{5}=\Gamma_{(1,2)}$. Note that $R_1\subseteq A\Gamma$ or $R_7\subseteq A\Gamma$. If $R_{1}\subseteq A\Gamma$, then $R_{2}=\Gamma_{(1,2)}$ since $p^{6}_{1,5}\neq0$, which is impossible; if $R_{7}\subseteq A\Gamma$, by $p^{1}_{2,2}\neq0$ from Lemma \ref{p1ij}, then $R_{7}=\Gamma_{(1,2)}$, a contradiction.

Suppose $R_{g/2}\subseteq A\Gamma$. Note that $g\in\{6,8\}$ and $R_{g/2}=\Gamma_{(1,1)}$. Since $p^{g-2}_{2,g-4}\neq0$, we have $R_2=\Gamma_{(1,2)}$. If $R_{g-1}\subseteq A\Gamma$, by $p^{1}_{2,2}\neq0$ from Lemma \ref{p1ij}, then $R_{g-1}=\Gamma_{(1,2)}$, which is impossible. Therefore, $R_{g-1}\nsubseteq A\Gamma$.

If $g=6$, then $R_1\subseteq A\Gamma$, which implies $R_1=\Gamma_{(1,2)}$ since $p_{2,3}^5\neq0$, a contradiction. Thus, $g=8$. If $R_3\subseteq A\Gamma$ or $R_5\subseteq A\Gamma$, by $p^{5}_{2,3}\neq0$ and $p^{3}_{2,2}\neq0$ from Lemma \ref{p322}, then $(y_3,x_3)$ or $(x_3,y_3)\in\Gamma_{(1,2)}$, which implies $R_3=\Gamma_{(1,2)}$ or $R_5=\Gamma_{(1,2)}$, contrary to the fact that $R_2=\Gamma_{(1,2)}$. Hence, $R_1\subseteq A\Gamma$. Since $p^{3}_{1,2}\neq0$ and $p^{5}_{1,4}\neq0,$ one gets $R_3=\Gamma_{(2,2)}$, contrary to the fact that $R_3$ is non-symmetric.
\end{proof}


\begin{lemma}\label{girth}
Suppose $k_1>k_g+1.$ If $\Gamma$ is not one of the digraphs in Theorem \ref{main} (i) and (iv), then $g$ is even.
\end{lemma}
\begin{proof}
Suppose for the contrary that $g$ is odd. According to  \cite[Theorem 3.3]{ln} and Theorem \ref{drdg}, we have $g\leq8$. It follows that $g=3$, $5$ or $7$. If $g=3$, then $\Gamma$ is one of the digraphs in Theorem \ref{main} (i) or (iv), a contradiction.
Thus, $g=5$ or $7$.

\textbf{Case 1.} $g=5.$

Since $\Gamma$ is strongly connected and not one of the digraphs in Theorem \ref{main} (i) or (iv), by Fact \ref{sub}, there exists $i\in\{2,3\}$ such that $R_i\subseteq A\Gamma$. Let $(x,y)\in R_i$. According to Lemma \ref{jcs} (i) and Lemma \ref{p322}, one gets $p_{3,3}^{2}=p^{3}_{2,2}\neq0$, and so $(x,y)\in \Gamma_{(1,2)}$, which imply $R_i=\Gamma_{(1,2)}$.

Let $(u,v)\in R_{1}$. By Lemma \ref{jcs} (i) and Lemma \ref{p1ij}, we have $p^{4}_{3,3}=p_{2,2}^1\neq0$ and $p^{4}_{2,2}=p^{1}_{3,3}\neq0$, which imply $\partial_{\Gamma}(v,u)\leq2$ and $\partial_{\Gamma}(u,v)\leq2$. Since $R_1$ is non-symmetric, we obtain $(u,v)\in\Gamma_{(1,2)}\cup\Gamma_{(2,1)}$, contrary to the fact that $R_i=\Gamma_{(1,2)}$.

\textbf{Case 2.} $g=7.$

Since $\Gamma$ is strongly connected and not one of the digraphs in Theorem \ref{main} (i) or (iv), by Fact \ref{sub}, there exists $i\in\{2,3,4,5\}$ such that $R_i\subseteq A\Gamma$.

Suppose $(R_2\cup R_5)\cap A\Gamma\neq\emptyset$. According to Fact  \ref{zhuanzhi}, we may assume $R_2\subseteq A\Gamma.$ Let $(u,v)\in R_{3}$. By Lemma \ref{p322}, one gets $p_{2,2}^3\neq0$, and so $\partial_{\Gamma}(u,v)\leq2$. Since $p_{2,2}^4\neq0$, one has $\partial_{\Gamma}(v,u)\leq2$. The fact that $R_3$ is non-symmetric implies $(u,v)\in\Gamma_{(1,2)}\cup\Gamma_{(2,1)}$. It follows that $R_3=\Gamma_{(1,2)}$ or $R_4=\Gamma_{(1,2)}$. Since $p^{5}_{2,3}\neq0$ and $k_5p^{5}_{2,4}=k_{3}p^{3}_{2,2}\neq0$ from Lemma \ref{jcs} (i), one has $R_2=\Gamma_{(1,2)}$, a contradiction.

Suppose $(R_2\cup R_5)\cap A\Gamma=\emptyset$. It is obvious that $(R_3\cup R_4)\cap A\Gamma\neq\emptyset$. By Fact   \ref{zhuanzhi}, we may assume $R_3\subseteq A\Gamma.$ Let $(u,v)\in R_{1}$. Since $p^{6}_{3,3}\neq0$ and $p^{1}_{3,3}\neq0$ from Lemma \ref{p1ij}, we get $\partial_{\Gamma}(v,u)\leq2$ and $\partial_{\Gamma}(u,v)\leq2$. Since $R_1$ is non-symmetric, one obtains $(u,v)\in\Gamma_{(1,2)}\cup\Gamma_{(2,1)}$, which implies $R_1=\Gamma_{(1,2)}$ or $R_6=\Gamma_{(1,2)}$. If $R_1=\Gamma_{(1,2)}$, by $p^{4}_{1,3}\neq0$, then $R_3=\Gamma_{(1,2)}$, a contradiction. Then $R_6=\Gamma_{(1,2)}$. Since $p_{3,6}^2=p_{1,4}^5\neq0$ and $p_{6,6}^5=p_{1,1}^2\neq0$ from Lemma \ref{jcs} (i), we have $R_2=\Gamma_{(2,2)}$, contrary to the fact that $R_2$ is non-symmetric.
\end{proof}

\begin{lemma}\label{R3R4}
Suppose $k_1>k_g+1$. If $g\in\{6,8\}$, then $(R_3\cup R_{g-3}\cup R_{g/2})\cap A\Gamma=\emptyset$.
\end{lemma}
\begin{proof}
Suppose not. By Facts \ref{sub} and \ref{zhuanzhi}, we may assume $R_3\subseteq A\Gamma$ or $R_{g/2}\subseteq A\Gamma$. Without loss of generality, we may assume $R_i\subseteq A\Gamma $ for some $i<d/2$. Let $(x_i,y_i)\in R_i$ for $1\leq i\leq 3$.

\textbf{Case 1.} $g=6$.

Note that $R_3\subseteq A\Gamma$. Then $R_3=\Gamma_{(1,1)}$. Since $R_i\subseteq A\Gamma$ for some $i<d/2$, from Lemma \ref{arc} and Fact \ref{sub}, one has $R_1\subseteq A\Gamma$ and $R_2\cap A\Gamma=\emptyset$, or $R_1\cap A\Gamma=\emptyset$ and $R_2\subseteq A\Gamma$. If $R_1\subseteq A\Gamma$ and $R_2\cap A\Gamma=\emptyset$, by $p_{1,1}^2p_{1,3}^4\neq0$, then $R_2=\Gamma_{(2,2)}$, contrary to the fact that $R_2$ is non-symmetric; if $R_1\cap A\Gamma=\emptyset$ and $R_2\subseteq A\Gamma$, by $p_{2,3}^5p_{2,2}^1\neq0$ from Lemma \ref{p1ij}, then $R_1=\Gamma_{(2,2)}$, a contradiction.

\textbf{Case 2.} $g=8$.

We divide it into two subcases according to whether $R_3$ is a subset of $A\Gamma$.

\textbf{Case 2.1.} $R_3\subseteq A\Gamma$.

Since $p^{6}_{3,3}\neq0$, one has $\partial_{\Gamma}(y_2,x_2)\leq2$. If $R_6\subseteq A\Gamma$, by Lemma \ref{jcs} (i) and Lemma \ref{p1ij}, then $p^{7}_{6,6}=p^{1}_{2,2}\neq0$ and $p^{1}_{3,6}=p^{7}_{2,5}\neq0$, which imply $R_1=\Gamma_{(2,2)}$ from Lemma \ref{arc}, a contradiction. Hence, $\partial_{\Gamma}(y_2,x_2)=2$, and so $R_6\subseteq \Gamma_2$.

Since $R_2$ is non-symmetric, we get $\partial_{\Gamma}(x_2,y_2)=1$ or $\partial_{\Gamma}(x_2,y_2)>2$. If $\partial_{\Gamma}(x_2,y_2)=1$, then $R_2=\Gamma_{(1,2)}$, and so $R_3=\Gamma_{(1,2)}$ since $p^{5}_{2,3}\neq0$, a contradiction. Hence, $\partial_{\Gamma}(x_2,y_2)>2$. By Lemma \ref{jcs} (i) and Lemma \ref{p322}, one has $k_2p^{2}_{3,6}=k_3p^{3}_{2,2}\neq0$. Since $R_6\subseteq\Gamma_2$, we get $\partial_{\Gamma}(x_2,y_2)=3$, and so $R_2=\Gamma_{(3,2)}$.

By Lemma \ref{p1ij}, we obtain $p^{1}_{3,3}\neq0$, and so $\partial_{\Gamma}(x_1,y_1)\leq2$. If $\partial_{\Gamma}(x_1,y_1)=1$, then $R_1\subseteq A\Gamma$, which implies $\partial_{\Gamma}(x_2,y_2)=2$ since $p^{2}_{1,1}\neq0$, contrary to the fact that $\partial_{\Gamma}(x_2,y_2)=3$. Thus, $\partial_{\Gamma}(x_1,y_1)=2$, and so $R_1\subseteq\Gamma_2$.

Since $R_1$ is non-symmetric, we have $\partial_{\Gamma}(y_1,x_1)=1$ or $\partial_{\Gamma}(y_1,x_1)>2$. Lemma \ref{jcs} (i) implies $p^{2}_{3,7}=p^{6}_{1,5}\neq0$. If $\partial_{\Gamma}(y_1,x_1)=1$, then $R_7\subseteq A\Gamma$, which implies $\partial_{\Gamma}(x_2,y_2)=2$, a contradiction. Therefore, $\partial_{\Gamma}(y_1,x_1)>2$. By Lemma \ref{jcs} (i) and Lemma \ref{p1ij}, one has $p^{7}_{3,6}=p^{1}_{2,5}\neq0$. Since $R_6=\Gamma_{(2,3)}$, we obtain $\partial_{\Gamma}(y_1,x_1)=3$, and so $R_1=\Gamma_{(2,3)}$, a contradiction.

\textbf{Case 2.2.} $R_3\nsubseteq A\Gamma$.

According to Fact \ref{sub}, we have $R_3\cap A\Gamma=\emptyset$. It follows that $R_4\subseteq A\Gamma$, and so $R_4=\Gamma_{(1,1)}$. Since $R_i\subseteq A\Gamma $ for some $i<d/2$, from Lemma \ref{arc} and Fact \ref{sub}, one has $R_1\subseteq A\Gamma$ and $R_2\cap A\Gamma=\emptyset$, or $R_1\cap A\Gamma=\emptyset$ and $R_2\subseteq A\Gamma$.

Suppose $R_1\subseteq A\Gamma$ and $R_2\cap A\Gamma=\emptyset$. Since $p^{2}_{1,1}p^{5}_{1,4}\neq0$, we have $\partial_{\Gamma}(x_2,y_2)=\partial_{\Gamma}(y_3,x_3)=2$ from Lemma \ref{arc}. It follows that $R_2\subseteq \Gamma_2$ and $R_5\subseteq \Gamma_2$. Since $R_2$ and $R_3$ are non-symmetric, one gets $\partial_{\Gamma}(y_2,x_2)>2$ and $\partial_{\Gamma}(x_3,y_3)>2$ from Lemma \ref{arc}. The fact that $p^{6}_{2,4}p^{3}_{1,2}\neq0$ implies $\partial_{\Gamma}(y_2,x_2)=\partial_{\Gamma}(x_3,y_3)=3$. Then $R_2=R_5=\Gamma_{(2,3)}$, a contradiction.

Suppose $R_1\cap A\Gamma=\emptyset$ and $R_2\subseteq A\Gamma$. In view of Lemma \ref{p322}, we obtain $p_{2,2}^1p_{2,2}^3\neq0$, and so $\partial_{\Gamma}(x_1,y_1)=\partial_{\Gamma}(x_3,y_3)=2$. It follows that $R_1\subseteq \Gamma_2$ and $R_3\subseteq \Gamma_2$. Since $R_1$ and $R_3$ are non-symmetric, one has $\partial_{\Gamma}(y_1,x_1)>2$ and $\partial_{\Gamma}(y_3,x_3)>2$ from Lemma \ref{arc}. The fact that $p^{7}_{3,4}p^{5}_{1,4}\neq0$ implies $\partial_{\Gamma}(y_1,x_1)=\partial_{\Gamma}(y_3,x_3)=3$, and so $R_1=R_3=\Gamma_{(2,3)}$, which is impossible.
\end{proof}

Now we are ready to prove Theorem \ref{main}.

\begin{proof}[Proof of Theorem \ref{main}]
In view of Lemma \ref{sur},  each digraph in Theorem \ref{main} (i)--(vi) is a weakly distance-regular digraph with $\mathfrak{X}$ as its attached scheme. Assume the contrary, namely, $\Gamma$ is not isomorphic to one of the digraphs in Theorem \ref{main} (i)--(vi).

First we consider the case $k_1=k_g+1$. Suppose $d=g-1$. Then $k_1=1$. Since $(X,R_1)$ is  distance-regular,  $(X,R_1)$ is a directed cycle, which implies $k_i=1$ for  $0\leq i\leq d$. It follows from \cite[Theorem 1.2]{HS04} that   $\Gamma$ is isomorphic to one of the digraphs in Theorem \ref{main} (i) and (iii), a contradiction. Suppose $d=g$. Then $(X,R_1)$ is a lexicographic product of a directed cycle $\Delta$ of length $g$ and  diameter $g-1$ by an empty graph. It follows that $\mathfrak{X}(\Delta)$ is a $P$-polynomial non-symmetric association scheme.  According to \cite[Theorem 2.1]{AM91}, $\Gamma$ is a lexicographic product of $\Gamma'$ by an empty graph or a lexicographic product of $\Gamma'$ by a complete graph, where $\Gamma'$ is a weakly distance-regular digraph with attached scheme $\mathfrak{X}(\Delta)$. By the case of $d=g-1$, $\Gamma'$ is isomorphic to one of the digraphs in Theorem \ref{main} (i) and (iii).  Therefore, $\Gamma$ is isomorphic to one of the digraphs in Theorem \ref{main} (i), (iii), (iv) and (vi), which is impossible.

Now, we  consider the case that $k_1\neq k_g+1$. In view of Theorem \ref{drdg}, one gets  $k_1>k_g+1$. By Facts \ref{sub} and  \ref{zhuanzhi}, we may assume $R_i\subseteq A\Gamma $ for some $i<d/2$. According to  \cite[Theorem 3.3]{ln} and Theorem \ref{drdg}, we have $g\leq8$. It follows from Lemma \ref{girth} that $g\in\{4,6,8\}$. If $g=4$, then $\Gamma$ is one of the digraphs in Theorem \ref{main} (i), (iii) and (iv) since $\mathfrak{X}$ is non-symmetric, a contradiction. Thus, $g=6$ or $8$.

\textbf{Case 1.} $g=6$.

Note that $R_i\subseteq A\Gamma$ for some $i<d/2$ and $\Gamma$ is not one of the digraphs in Theorem \ref{main} (i)--(vi). By Lemma \ref{R3R4} and Fact \ref{sub}, we have $R_1\cup R_4\subseteq A\Gamma$ or $R_2\cup R_5\subseteq A\Gamma$. If $R_1\cup R_4\subseteq A\Gamma$, then $R_1=R_4=\Gamma_{(1,2)}$ since $p^{5}_{1,4}p^{2}_{1,1}\neq0$; if $R_2\cup R_5\subseteq A\Gamma$, by $p^{4}_{2,2}\neq0$ and $p^{1}_{2,2}\neq0$ from Lemma \ref{p1ij}, then $R_2=R_5=\Gamma_{(1,2)}$, a contradiction.

\textbf{Case 2.} $g=8$.

Note that $R_i\subseteq A\Gamma$ for some $i<d/2$ and $\Gamma$ is not one of the digraphs in Theorem \ref{main} (i)--(vi). By Lemmas \ref{arc}, \ref{R3R4} and Fact \ref{sub}, we get $R_1\subseteq A\Gamma$ and $R_2\cap A\Gamma=\emptyset$, or $R_1\cap A\Gamma=\emptyset$ and $R_2\subseteq A\Gamma$.
If $R_1\subseteq A\Gamma$ and $R_2\cap A\Gamma=\emptyset$, from Lemma \ref{R3R4}, then $R_6\subseteq A\Gamma$, which implies $R_1=\Gamma_{(1,2)}$ and $R_6=\Gamma_{(1,2)}$ since $p_{1,6}^7p_{1,1}^2\neq0$, a contradiction. Hence, $R_1\cap A\Gamma=\emptyset$ and $R_2\subseteq A\Gamma$. By Lemma \ref{R3R4}, we have $R_7\subseteq A\Gamma$. In view of Lemma \ref{jcs} (i), one gets $p^{6}_{7,7}=p^{2}_{1,1}\neq0$, and so $R_2=\Gamma_{(1,2)}$. According to Lemma \ref{p1ij}, we have $p_{2,2}^1\neq0$, which implies $R_7=\Gamma_{(1,2)}$, a contradiction.
\end{proof}

\section*{Acknowledgements}

The authors are indebted to the anonymous reviewers for their useful comments and suggestions. We would like to thank Professor Hiroshi Suzuki for his suggestions. Y.~Yang is supported by NSFC (12101575) and the Fundamental Research Funds for the Central Universities (2652019319),  K. Wang is supported by the National Key R$\&$D Program of China (No.~2020YFA0712900) and NSFC (12071039, 12131011).

\end{CJK*}

\end{document}